\documentclass[12pt,a4paper]{amsart}
\usepackage{amssymb,amsmath,amsfonts}
\usepackage[usenames,dvipsnames,svgnames,table]{xcolor}
\usepackage{amsmath,amsfonts,amsthm}
\usepackage{amssymb,latexsym}
\usepackage{setspace}
\usepackage{ams thm,ams symb,amsmath}
\usepackage[all]{xy}
 \usepackage{graphicx}
\usepackage{tikz}
\theoremstyle{plain}
\usepackage[square,numbers]{natbib}
\headsep=1.2500cm \numberwithin{equation}{section}
\usepackage[left=3cm,right=3cm,top=2cm,bottom=2cm]{geometry} % page settings
\usepackage{amssymb, amsmath, latexsym, amsfonts, amsthm, mathrsfs}
\newtheorem{Theorem}{Theorem}[section]

\theoremstyle{remark}

\begin{document}
\begin{center}
{\large\bf Non-reflexivity of the Banach space $\Lambda{\rm B}V^{(p)}$}
\vskip.60in

{Batoul S. Mortazavi-Samarin$^{1}$, Mehdi Rostami$^{1*}$}\footnotetext{Corresponding author} \\[2mm]

{\footnotesize \textit{$^{1}$Department of Mathematics and Computer Science, Amirkabir University of Technology\\ (Tehran
Polytechnic), Iran}\\
}
\end{center}
\vskip 5mm
%-------------------------------------------------------------
%%%%%%%%%%%%%%%%%%%%%%%%%%%%%%%%%%%%%%%%%%%%%%%%%%%%%%%%%%%%%%%%%%%%%%%%%
\noindent{\bf Abstract.}
In this paper, we present a surprising finding that demonstrates the non-reflexivity of the Waterman-Shiba space.
Previously, Prus-Wi$\rm\acute{s}$niowski and Ruckle
\cite{1} extended the classical result showing that the space of functions of bounded variation is non-reflexive.
Our work offers a refinement and strengthening of this earlier result.
%%%%%%%%%%%%%%%%%%%%%%%%%%%%%%%%%%%%%%%%%%%%%%%%%%%%%%%%%%%%%%%%%%%%%%%%%
%%%%%%%%%%%%%%%%%%%%%%%%%%%%%%%%%%%%%%%%%%%%%%%%%%%%%%%%%%%%%%%%%%%%%%%%%
 \vskip.15in
\footnotetext { \textbf{2020 Mathematics Subject Classification}: 26A45, 26B30, 26A15}
\footnotetext { \textbf{Keywords}: Waterman sequence, Waterman-Shiba class, bounded variation, reflexivity}
\footnotetext{\textbf{E-mail:}{
bmortazavi@aut.ac.ir (B. S. Mortazavi-Samarin); mross@aut.ac.ir (M. Rostami)}}
\section{Introduction and Preliminaries}
The notion of $\Lambda$-variation was introduced by Waterman \cite{Waterman} and has become an important tool in analysis,
especially in the study of Fourier series,
where it provides a natural extension of the classical Dirichlet--Jordan criterion. A sequence $\Lambda = (\lambda_i)$ of
positive numbers is called a Waterman sequence
if it is nondecreasing and satisfies the divergence condition $\sum_{i=1}^{\infty} \frac{1}{\lambda_i} = \infty$.
For a real-valued function $f$ on $[0,1]$, the $\Lambda$-variation
$V_\Lambda(f)$ is defined as the supremum of sums of normalized increments of $f$ taken over suitable families of intervals.
The collection of all functions with finite
$\Lambda$-variation forms the space $\Lambda\mathrm{BV}$, which becomes a Banach space when equipped with the norm
$\|f\| = |f(0)| + V_\Lambda(f)$. In the special case
$\Lambda = (1)$, one recovers the classical space of functions of bounded variation, $\mathrm{BV}$.

The Waterman--Shiba class $\Lambda\mathrm{BV}^{(p)}$, defined for $1 \le p < \infty$, extends the Waterman class
by incorporating a $p$-norm into the definition of variation,
much like the way the spaces $\ell^p$ generalize $\ell^1$. This analogy suggests a close connection between
$\Lambda\mathrm{BV}^{(p)}$ and the sequence spaces $\ell^p$.
However, although the classical $\ell^p$ spaces are reflexive for $1 < p < \infty$, the reflexivity of
$\Lambda\mathrm{BV}^{(p)}$ is not immediately apparent due to the more delicate structure imposed by the $\Lambda$-variation.

Previously, Prus-Wi\'{s}niowski and Ruckle \cite{1} established that $\Lambda\rm BV$ is non-reflexive, drawing on the
analogy between $\rm BV$ and $\ell^1$ to extend the
classical non-reflexivity result for bounded variation functions. In this paper, we extend these ideas to the
Waterman-Shiba class $\Lambda {\rm B}V^{(p)}$ and obtain a
surprising result: despite the analogy with the reflexive $\ell^p$ spaces, $\Lambda{\rm B}V^{(p)}$ is non-reflexive.
This result highlights a striking difference between the
behavior of sequence spaces and their function space analogues under $\Lambda$-variation,
and it refines previous results by providing a deeper understanding of the structure of these generalized bounded variation spaces.

For a Waterman sequence $\Lambda = (\lambda_i)$, we stand
$\Lambda(r)$ for $ \sum_{i=1}^{[r]}\frac{1}{\lambda_i}$, where
$r\geq1$. For $I= [a,b]\subseteq[0,1]$ and $p\in[1,\infty)$, we
denote $$f(I) := f(b) - f(a),\qquad V_{\Lambda,p}(f) :=
\underset{n}{\sup} \Big(\sum_{i=1}^n
\frac{|f(I_i)|^p}{\lambda_i}\Big)^{1/p},$$ where the supremum is
taken over all finite families  $\{I_i\}_{i=1}^\infty$ of
non-overlapping closed intervals with endpoint in $[0,1]$. Let $f$
be a real valued function on $[0,1]$ (in general, all discussions
for real valued functions on an arbitrary closed interval remain
valid). If $V_{\Lambda,p}(f) < \infty$, $f$ is called of bounded
$p$-$\Lambda$-variation. The linear space of all functions of
bounded $p$-$\Lambda$-variation is called the Waterman-Shiba space,
denoted by $\Lambda{\rm B}V^{(p)}$. This class was introduced by Shiba in 1980 \cite{2}. $\Lambda{\rm B}V^{(p)}$ is a Banach space
under the following norm $$\|f\|_{\Lambda,p} = |f(0)|+
V_{\Lambda,p}(f).$$ In the case $p=1$, $\Lambda\rm BV^{(p)}$
coincides with the Waterman space.
%%%%%%%%%%%%%%%%%%%%%%%%%%%%%%%%%%%%%%%%%%%%%%%%%%%%%%%%%%%%%%%%%%%%%%%%%%%%%%%%%%%%%%%%%%%%%%%%%%%%%%%%%%%%%%%%%%%%%%%%%%%%%%%%%%%%
\section{Main Result}
The proof technique of \cite{1} motivated us to extend that result. Through, $\Lambda$
is a  Waterman sequence and since $\Lambda(r)$ tends to $\infty$ when $ r\rightarrow \infty$, we may, without loss of generality, assume that $ \lambda_1 =1$.

\begin{Theorem}
Let  $p\in (1,\infty)$ and $\Lambda$ be a Waterman sequence. Then the Waterman-Shiba space $\Lambda{\rm B}V^{(p)}$ is non-reflexive.
\end{Theorem}
\begin{proof}
Let  $q\in(1,\infty)$ be chosen such that $ 1/p+1/q=1$. We now construct a family of functions which is practical in the embedding theory, see \cite{-1} and \cite{3}.
Let  $b_{n,j}:=2^{-n}+\frac{2j-2}{2^{4n}}$, $c_{n,j}:=2^{-n}+\frac{2j-1}{2^{4n}} $ and $M_n = 2^{3n-1}$. Let $h_1$ be the zero function on $[0,1]$ and
for each $n=2, 3,\dots$, we define the functions $h_n$  by
\begin{equation*}
\begin{split}
h_n(y) := \left\{
\begin{array}{rl}
2^{-n}{\big(\Lambda(M_n) \big)}^{\frac{-1}{q}}\qquad& \qquad y \in [b_{n,j},c_{n,j});\ 1\leq j \leq \frac{M_n-2}{2} , \\
\\
0\qquad \qquad\qquad\ & \qquad {\rm otherwise}.
\end{array} \right.
\end{split}
\end{equation*}
Since the supports of the functions $h_n$ are disjoint, the function $h(x) :=  \sum_{n=1}^{\infty}h_n(x)$ is well-defined. Furthermore, we estimate
\begin{equation*}
\begin{split}
V_{\Lambda,q}(h) \leq \sum_{n=1}^\infty V_{\Lambda,q}(h_n)
&\leq \sum_{n=1}^\infty \Big( \sum_{j=1}^{2\frac{M_n-2}{2}} \frac{\big(2^{-n}{\big(\Lambda(M_n) \big)}^{\frac{-1}{q}}\big)^q}{\lambda_j}\Big)^{1/q}\\
&\leq \sum_{n=1}^\infty \Big( \sum_{j=1}^{M_n} \frac{\big(2^{-n}{\big(\Lambda(M_n) \big)}^{\frac{-1}{q}}\big)^q}{\lambda_j}\Big)^{1/q}
\\&\leq  \sum_{n=1}^\infty 2^{-n} \Big( \frac{  \Lambda(M_n) }{\Lambda(M_n)}\Big)^{1/q}
= 1.
\end{split}
\end{equation*}
 Next, for any fixed $n> 1$, a direct calculation shows that
\begin{equation*}
\begin{split}
c_{n+1,\frac{M_{n+1}-2}{2}}&=2^{-n-1}+\frac{2\times \frac{M_{n+1}-2}{2}-1}{2^{4n+4}} = \frac{2^{3n+3}+ 2^{3n+2}-3}{2^{4n+4}}\leq 2^{-n}= b_{n,1}\leq  b_{n,\frac{M_{n}-2}{2}}.
\end{split}
\end{equation*}
This produces a strictly decreasing sequence
$$ \cdots\lvertneqq b_{4,\frac{M_4-2}{2}}\lvertneqq   c_{4,\frac{M_4-2}{2}}\lvertneqq   b_{3,\frac{M_3-2}{2}}\lvertneqq  c_{3,\frac{M_3-2}{2}}\lvertneqq   b_{2,\frac{M_2-2}{2}}\lvertneqq   c_{2,\frac{M_2-2}{2}},$$
which converges to zero. We denote this sequence by $(r_i)_{i=0}^\infty$. Equivalently the sequence, $(r_i)_{i=0}^\infty$ may be described as follows:
\begin{equation*}
\begin{split}
r_i:=\left\{
\begin{array}{rl}
c_{{k_i}+2,\frac{M_{k_i+2}-2}{2}}\qquad\qquad& \qquad {\rm{for~ an~ even}} ~i~ {\rm{where}} ~i=2{k_i}, \\
\\
b_{{k_i}+2,\frac{M_{k_i+2}-2}{2}}\qquad \qquad & \qquad{\rm{for~ an~ odd}} ~i~ {\rm{where}}~ i=2{k_i}+1,
\end{array} \right.
\end{split}
\end{equation*}
for $i=0,1,2,\dots$. Define the intervals $J_i = [r_{i},r_{i-1}]$.
This yields a family of pairwise disjoint intervals $\mathcal{J} = \{J_i\}_{i=1}^\infty$.
Now consider the interval $J^\prime := [\frac{3}{8},\frac{3}{8}+\frac{1}{2^8})$. One can verify that
$J^\prime\subset [0,1]\backslash \bigcup^\infty_{i=1} J_i$.  Indeed,
$$\frac{3}{8} = \frac{1}{4} + \frac{2^5}{2^8} > \frac{1}{4} + \frac{2^5-3}{2^8} = c_{2,\frac{M_{2}-2}{2}}.$$
For convenience, define $N_k:=2k$ for all integers $k\geq0$.
For $f\in \Lambda {\rm B}V^{(p)}$ and a given increasing sequence $(n_k) $ of positive integers, we define the functional
$$L_{(n_k)}(f) = f(J^\prime)+ \sum_{j=1}^\infty p_{(n_k)}(j) \sum_{i=1}^{N_{j} - N_{{j}-1}}{(-1)}^i \frac{ f(J_{N_{j-1}+i})|h(J_{N_{j-1}+i})|}{\lambda_{ N_{{j}-1}+i}}, $$
where
\begin{equation*}
\begin{split}
p_{(n_k)}(j) = \left\{
\begin{array}{rl}
0\qquad \qquad\qquad &{\rm{if~}}j\neq n_k ~{\rm{for~all}}~ k,\\
1\qquad \qquad\qquad&{\rm{if~}}j= n_k ~{\rm{for~an ~odd}}~ k,\\
-1\qquad \qquad\qquad&{\rm{if~}}j= n_k~ {\rm{for~an~even}}~ k.
\end{array} \right.
\end{split}
\end{equation*}
It is noteworthy that
\begin{equation*}
\begin{split}
\sum_{i=1}^{N_{j} - N_{{j}-1}}\Big|{(-1)}^i \frac{ f(J_{N_{j-1}+i})|h(J_{N_{j-1}+i})|}{\lambda_{ N_{{j}-1}+i}}{\Big|}
&= \sum_{i=1}^{2}\Big| \frac{ f(J_{N_{j-1}+i})|h(J_{N_{j-1}+i})|}{\lambda_{ N_{{j}-1}+i}}{\Big|}
\\&=\Big| \frac{ f(J_{N_{j-1}+1})|h(J_{N_{j-1}+1})|}{\lambda_{ N_{{j}-1}+1}}{\Big|} + \Big| \frac{ f(J_{N_{j-1}+2})|h(J_{N_{j-1}+2})|}{\lambda_{ N_{{j}-1}+2}}{\Big|}
\\&=\Big| \frac{ f(J_{N_{j-1}+1})|h(J_{N_{j-1}+1})|}{\lambda_{ N_{{j}-1}+1}}{\Big|} + \Big| \frac{ f(J_{N_{j}})|h(J_{N_{j}})|}{\lambda_{ N_{{j}}}}{\Big|}.
\end{split}
\end{equation*}
So
\begin{equation*}
\begin{split}
&\Big|\sum_{j=1}^\infty p_{(n_k)}(j)    \sum_{i=1}^{N_{j} - N_{{j}-1}}{(-1)}^i \frac{ f(J_{N_{j-1}+i})|h(J_{N_{j-1}+i})|}{\lambda_{ N_{{j}-1}+i}}\Big|
\\&\leq
\sum_{j=1}^\infty {\Big|} p_{(n_k)}(j)  \sum_{i=1}^{N_{j} - N_{{j}-1}}{(-1)}^i \frac{ f(J_{N_{j-1}+i})|h(J_{N_{j-1}+i})|}{\lambda_{ N_{{j}-1}+i}}{\Big|}
\\&\leq \sum_{j=1}^\infty   \sum_{i=1}^{N_{j} - N_{{j}-1}} {\Big|}\frac{ f(J_{N_{j-1}+i})|h(J_{N_{j-1}+i})|}{\lambda_{ N_{{j}-1}+i}}{\Big|}
\\&=
\sum_{i=1}^{N_{1} - N_{{0}}} {\Big|}\frac{f(J_{N_{0}+i})|h(J_{N_{0}+i})|}{\lambda_{ N_{{0}}+i}}{\Big|} +\sum_{i=1}^{N_{2} - N_{{1}}} {\Big|}\frac{ f(J_{N_{1}+i})|h(J_{N_{1}+i})|}{\lambda_{ N_{{1}}+i}}{\Big|} +\cdots
\\&=
{\Big|}\frac{f(J_{N_{0}+1})|h(J_{N_{0}+1})|}{\lambda_{ N_{{0}}+1}}{\Big|}+{\Big|}\frac{f(J_{N_{1}})|h(J_{N_{1}})|}{\lambda_{ N_{{1}}}}{\Big|}
+
{\Big|}\frac{f(J_{N_{1}+1})|h(J_{N_{1}+1})|}{\lambda_{ N_{{1}}+1}}{\Big|}+
{\Big|}\frac{f(J_{N_{2}})|h(J_{N_{2}})|}{\lambda_{ N_{{2}}}}{\Big|}+\cdots
\\&
\leq \sum_{i=1}^\infty{\Big|}\frac{ f(J_{i})|h(J_{i})|}{\lambda_{i}}{\Big|}.
\end{split}
\end{equation*}
Using H$\rm\ddot{o}$lder's inequality with conjugate exponents $p$ and
$q$, we obtain
\begin{equation*}
\begin{split}
&\sum_{j=1}^\infty p_{(n_k)}(j) \sum_{i=1}^{N_{j} - N_{{j}-1}}{(-1)}^i \frac{ f(J_{N_{j-1}+i})|h(J_{N_{j-1}+i})|}{\lambda_{ N_{{j}-1}+i}}\leq
\sum_{i=1}^\infty{\Big|}\frac{ f(J_{i})|h(J_{i})|}{\lambda_{i}}{\Big|}
\\
&\leq {\left(\sum_{i=1}^{\infty} {\left(\frac{ | f({J_{i}}) |  }{{\lambda_{i}}^{1/p}}\right)}^{p}\right)}^{1/p}
{\left(\sum_{i=1}^{\infty} {\left(\frac{ | h(J_{i}) |  }{{\lambda_{i}}^{1/q}}\right)}^{q}\right)}^{1/q}\\
&\leq  V_{\Lambda,p}(f)V_{\Lambda,q}(h).
\end{split}
\end{equation*}
This shows that the series $$\sum_{j=1}^\infty p_{(n_k)}(j) \sum_{i=1}^{N_{j} - N_{{j}-1}}{(-1)}^i \frac{ f(J_{N_{j-1}+i})|h(J_{N_{j-1}+i})|}{\lambda_{ N_{{j}-1}+i}}$$
is absolutely convergent and particularly convergent for every function $f\in \Lambda {\rm B}V^{(p)}$. Consequently, the functional $L_{(n_k)}$ is well-defined and we may freely reorder the terms in this series. We have
\begin{equation*}
\begin{split}
L_{(n_k)}(\alpha f
+ g ) &= (\alpha f
+ g ) (J^\prime)\\& + \sum_{j=1}^\infty p_{n_k}(j)  \sum_{i=1}^{N_{n_k} - N_{n_{k}-1}}{(-1)}^i \frac{ (\alpha f
+ g )(J_{N_{n_{k}-1}+i})|h(J_{N_{j-1}+i})|}{\lambda_{N_{n_{k}-1}+i}}\\
&= (\alpha f
+ g ) (b)  - (\alpha f
+ g ) (a)\\& +\sum_{j=1}^\infty p_{n_k}(j)  \sum_{i=1}^{N_{n_k} - N_{n_{k}-1}}{(-1)}^i \frac{ \alpha f (J_{N_{n_{k}-1}+i})|h(J_{N_{j-1}+i})| + g(J_{N_{n_{k}-1}+i})|h(J_{N_{j-1}+i})|}{\lambda_{i}}
\\&=\alpha f
(b)  - \alpha f
(a) +\sum_{j=1}^\infty p_{n_k}(j)   \sum_{i=1}^{N_{n_k} - N_{n_{k}-1}}{(-1)}^i \frac{ \alpha f(J_{N_{n_{k}-1}+i})|h(J_{N_{j-1}+i})|}{\lambda_{i}}
\\&+ g (b)  - g  (a) +\sum_{j=1}^\infty p_{n_k}(j)  \sum_{i=1}^{N_{n_k} - N_{n_{k}-1}}{(-1)}^i \frac{  g(J_{N_{n_{k}-1}+i})|h(J_{N_{j-1}+i})|}{\lambda_{i}}\\
&= \alpha L_{(n_k)}( f ) + L_{(n_k)}(g ),
\end{split}
\end{equation*}
where $a=\frac{3}{8}$ and $b=\frac{3}{8}+\frac{1}{2^8}$.
On the other hand, take any finite family of intervals $\{I_i\}_{i=1}^n$,
with $I_1:= J^\prime$ and $I_2,\dots, I_n$ are some non-overlapping intervals contained in $[0,1]\backslash J^\prime$. For $f\in\Lambda{\rm B}V^{(p)}$ with $\|f\|_{\Lambda,p}\leq1$, we have $\dfrac{|f(J^{\prime})|^p}{\lambda_1}=|f(J^{\prime})|^p\leq1$ and so $|f(J^{\prime})|\leq1$.
Therefore,
$$|L_{(n_k)}(f)|\leq |f(J^{\prime})| + V_{\Lambda,p}(f)V_{\Lambda,q}(h) \leq 1+\|f\|_{\Lambda,p}\|h\|_{\Lambda,q} .$$
Since $h\in \Lambda {\rm B}V^{(q)}$, $V_{\Lambda,q}(h)\leq1$, and $h(0)=0$, it follows that
$$\|L_{(n_k)}\|=\sup_{\|f\|_{\Lambda,p}\leq1}|L_{(n_k)}(f)|\leq \|h\|_{\Lambda,q}=|h(0)|+V_{\Lambda,q}(h)\leq1.$$
Hence, $L_{(n_k)}$ is a bounded linear functional on $\Lambda{\rm
B}V^{(p)}$ with operator norm at most $1$. For each fixed $l$, we define a function $f_l$ which vanishes outside $(r_{{N_{l}}},r_{{N_{l-1}}})\bigcup J^\prime$. On
$J^{\prime}$ the function takes the constant value $p_{(n_k)}(l)$ on
$J^\prime$. On $(r_{{N_{l}}},r_{{N_{l-1}}+1})$ and
$(r_{{N_{l-1}+1}},r_{{N_{l-1}}})$, it is defined linearly so that
$f_l(r_{N_{l-1}+1}):=1.$ For instance, in the case where $l\neq n_k$
for all $k$, the graph of $f_l$  has the shape illustrated as
follows:
\begin{figure}[!ht]
\centering
\begin{tikzpicture}
\def\l{1.2}
\coordinate (o) at (0,0);
\coordinate (b) at (3-\l,0);
\coordinate (A) at (3,{\l*sqrt(3)});
\coordinate (a) at (3+\l,0);
 \draw[<->] (0,{\l*sqrt(3)+.5}) node[above] {$y$}|- (6,0) node[right] {$x$} node[midway, below left] {$0$};
\draw (b) node[below] {\strut $r_{N_l}$} --(A) node[above]{${}$} --(a)node[below] {\strut $r_{N_{l-1}}$};
\draw[densely dashed] (o|-A) node[left]{$1.0$} -| (o-|A) node[below]{\strut $r_{N_{l-1}+1}$};
\end{tikzpicture}
\end{figure}
\par
Since $\lambda_1=1$ and the sequence $(\lambda_i)$ is increasing, we
obtain
\begin{equation*}
\begin{split}
\| f_l \|_{\Lambda,p}&\leq {{\left( \sum_{k=1}^3\frac{|f(I^\prime_k)|^p}{\lambda_k} \right)}}^{1/p}\\&
\leq {{\left( \sum_{k=1}^3\frac{|2|^p}{\lambda_k} \right)}}^{1/p}
\\& ={{\left( \sum_{k=1}^3\frac{2\times{\lambda_1}}{\lambda_k} \right)}}^{1/p}\\&
\leq  {{\left( \sum_{k=1}^3 \frac{2\times{\lambda_k}}{\lambda_k} \right)}}^{1/p}\\& = 6^{1/p}.
\end{split}
\end{equation*}
Thus, $\{f_l\}$
is a bounded sequence in $\Lambda{\rm B}V^{(p)}$ with uniform bound $6^{1/p}$.

Now, consider an arbitrary subsequence
$ (f_{n_s}) $ of the  sequence $(f_{l})$. Since $f_l=0$ outside the interval $(r_{{N_{l}}},r_{{N_{l-1}}})$, we have $ f_{l}(J_{N_j}) =0 $ whenever $l\neq j$.
Therefore, for every natural number $s$, we have
\begin{equation*}
\begin{split}
&L_{(n_k)}(f_{n_s})
= f_{n_s}(J^\prime) + \sum_{j=1}^\infty p_{(n_k)}(j)\sum_{i=1}^{N_{j} - N_{j-1}}~ {(-1)}^i \frac{ f_{n_s}(J_{N_{j-1}+i}) |h(J_{N_{j-1}+i})|}{\lambda_{N_{j-1}+i}}
\\&= f_{n_s}(J^\prime) +  p_{(n_k)}(n_s)\sum_{i=1}^{N_{n_s} - N_{n_s-1}}~ {(-1)}^i \frac{ f_{n_s}(J_{N_{n_s-1}+i}) |h(J_{N_{{n_s}-1}+i})|}{\lambda_{N_{n_s-1}+i}}
\\&+\sum_{\underset{j\neq n_s}{j=1}}^\infty p_{(n_k)}(j)\sum_{i=1}^{N_{j} - N_{j-1}}~ {(-1)}^i \frac{ f_{n_s}(J_{N_{j-1}+i}) |h(J_{N_{j-1}+i})|}{\lambda_{N_{j-1}+i}}
\\&= f_{n_s}(J^\prime) + p_{(n_k)}(n_s)\sum_{i=1}^{N_{n_s} - N_{n_{s}-1}}~ {(-1)}^i \frac{ f_{n_s}(J_{N_{n_{s}-1}+i})|h(J_{N_{{n_s}-1}+i})|}{\lambda_{N_{n_{s}-1}+i}}
\\& =  p_{(n_k)}(n_s) + p_{(n_k)}(n_s)\sum_{i=1}^{N_{n_s} - N_{n_{s}-1}}~ {(-1)}^i \frac{ f_{n_s}(J_{N_{n_{s}-1}+i})|h(J_{N_{{n_s}-1}+i})|}{\lambda_{N_{n_{s}-1}+i}}
\\&=  p_{(n_k)}(n_s) {\Big(}1 + \sum_{i=1}^{N_{n_s} - N_{n_{s}-1}}~  \frac{ |h(J_{N_{{n_s}-1}+i})|}{\lambda_{N_{n_{s}-1}+i}}{\Big)}
\\&= (-1)^{s+1} {\Big(}1 + \sum_{i=1}^{N_{n_s} - N_{n_{s}-1}}~  \frac{ |h(J_{N_{{n_s}-1}+i})|}{\lambda_{N_{n_{s}-1}+i}}{\Big)}.
\end{split}
\end{equation*}
Since
\begin{equation*}
\begin{split}
&\sum_{i=1}^{N_{n_s} - N_{n_{s}-1}}~ {(-1)}^i \frac{ f_{n_s}(J_{N_{n_{s}-1}+i})|h(J_{N_{{n_s}-1}+i})|}{\lambda_{N_{n_{s}-1}+i}}
\\&= - \frac{ f_{n_s}(J_{N_{n_{s}-1}+1})|h(J_{N_{{n_s}-1}+1})|}{\lambda_{N_{n_{s}-1}+1}} +  \frac{ f_{n_s}(J_{N_{n_{s}-1}+2})|h(J_{N_{{n_s}-1}+2})|}{\lambda_{N_{n_{s}-1}+2}}
\\&=- \frac{ \big(f_{n_s}(r_{N_{n_{s}-1}+1-1}) - f_{n_s}(r_{N_{n_{s}-1}+1})\big)|h(J_{N_{{n_s}-1}+1})|}{\lambda_{N_{n_{s}-1}+1}}
\\& +  \frac{ \big(f_{n_s}(r_{N_{n_{s}-1}+2-1}) - f_{n_s}(r_{N_{n_{s}-1}+2})\big)|h(J_{N_{{n_s}-1}+2})|}{\lambda_{N_{n_{s}-1}+2}}
\\&=- \frac{ \big(f_{n_s}(r_{N_{n_{s}-1}}) - f_{n_s}(r_{N_{n_{s}-1}+1})\big)|h(J_{N_{{n_s}-1}+1})|}{\lambda_{N_{n_{s}-1}+1}}
\\& +  \frac{ \big(f_{n_s}(r_{N_{n_{s}-1}+1}) - f_{n_s}(r_{N_{n_{s}-1}+2})\big)|h(J_{N_{{n_s}-1}+2})|}{\lambda_{N_{n_{s}-1}+2}}
\\&- \frac{ \big(0 - 1\big)|h(J_{N_{{n_s}-1}+1})|}{\lambda_{N_{n_{s}-1}+1}}
+  \frac{ \big(1 - 0)\big)|h(J_{N_{{n_s}-1}+2})|}{\lambda_{N_{n_{s}-1}+2}}
\\& =\sum_{i=1}^{N_{n_s} - N_{n_{s}-1}}~  \frac{ |h(J_{N_{{n_s}-1}+i})|}{\lambda_{N_{n_{s}-1}+i}},
\end{split}
\end{equation*}
then,
    $$L_{(n_k)}(f_{n_s})~   > 1   \qquad\qquad{\rm{for~ an ~odd}}\ \ s,
    $$ and
    $$L_{(n_k)}(f_{n_s})~  < -1\qquad\quad\ {\rm{for~an~even}}\ \ s. $$
    This implies that the sequence $\{f_{n_s}\}$ does not converge weakly. By \cite[Corollary 2.8.9]{0},
    this implies that the Waterman-Shiba space $\Lambda{\rm B}V^{(p)}$ is not reflexive.
\end{proof}
The non-reflexivity of $\Lambda{\rm B}V^{(p)}$ has important
implications in Fourier analysis and related areas. Reflexivity is
closely tied to weak compactness properties, duality, and the
structure of linear operators, all of which play a key role in the
convergence of Fourier series and the analysis of function
approximations. Understanding the reflexivity, or lack thereof, of
function spaces like $\Lambda{\rm B}V^{(p)}$ helps clarify the
limitations of certain analytical methods, informs the design of
approximation schemes, and provides insight into the behavior of
operators on these spaces. Consequently, our result not only
advances the theoretical understanding of generalized bounded
variation spaces but also offers a foundation for future
applications in harmonic analysis and functional analysis.
%\newpage
%\noindent
\section {Conclusion}

\noindent In this paper, we establish an unexpected result
concerning the structure of the Waterman--Shiba space, namely, its
non-reflexivity. Earlier, Prus-Wi$\rm\acute{s}$niowski and Ruckle
\cite{1} extended the classical theorem asserting the
non-reflexivity of the space of functions of bounded variation. The
result obtained here sharpens and completes their work by showing
that non-reflexivity persists even within the broader framework of
the Waterman--Shiba classes.

\noindent
{\bf Data Availability}\\
Data sharing is not applicable to this article as no new data were created or analyzed in this study.

\noindent
{\bf Conflicts of Interest}\\
The authors declare that there are no conflicts of interest
regarding the publication of this paper.

\noindent
{\bf Authors' Contributions}\\
The authors contributed equally to this work.

%%%%%%%%%%%%%%%%%%%%%%%%%%%%%%%%%
\begin{small}

\end{small}
\end{document}